\documentclass[11pt]{article}

\usepackage[a4paper,margin=1in]{geometry}
\usepackage[T1]{fontenc}
\usepackage[utf8]{inputenc}
\usepackage{lmodern}
\usepackage{amsmath,amssymb,amsthm,mathtools}
\usepackage{graphicx}
\usepackage{float}
\usepackage{microtype}
\usepackage{enumitem}
\usepackage{caption}
\usepackage{tikz}
\usepackage[hidelinks]{hyperref}
\usepackage{cite}
\usepackage{authblk}

\allowdisplaybreaks

\newtheorem{theorem}{Theorem}[section]
\newtheorem{lemma}[theorem]{Lemma}
\newtheorem{proposition}[theorem]{Proposition}
\newtheorem{corollary}[theorem]{Corollary}
\newtheorem{problem}[theorem]{Problem}
\theoremstyle{remark}
\newtheorem{remark}[theorem]{Remark}

\newcommand{\N}{\mathbb{N}}
\newcommand{\Reach}{\operatorname{Reach}}
\newcommand{\cR}{\mathcal{R}}

\numberwithin{equation}{section}

\title{The Exact End-Degree Threshold for\\
Finite-Width Directed Hexagonal Grids}

\author[1]{Dan Zhu}
\author[2]{Zhenhua Lyu\thanks{Corresponding author. Email: lyuzhh@outlook.com (Lyu)}}

\affil[1]{School of Mathematical Sciences, Shenzhen University,
Shenzhen 518060, China}

\affil[2]{School of Science, Shenyang Aerospace University,
Shenyang 110136, China}

\date{August 21, 2026}

\begin{document}
\maketitle
\footnotetext[2]{Email: zd2609477480@163.com (Zhu)}
\vspace{-2.5em}

\begin{abstract}
Grid theorems provide a fundamental link between the structure of ends and the existence of grid-like subgraphs in infinite graphs and digraphs. For every fixed width $n$, let $k(n)$ denote the least positive integer such that every digraph with an end of in-degree at least $k(n)$ contains a subdivision of the directed hexagonal grid of width $n$. Hamann and Heuer asked the exact vaule of $k(n)$. We solve this problem by proving that
$$
k(n)=\left\lfloor\frac{3n}{2}\right\rfloor-1
\qquad\text{for every }n\geq4,
$$
while $k(1)=1$, $k(2)=2$, and $k(3)=4$. For the upper bound, we establish a finite vacancy theorem for labelled token slides, convert it into a closed directed schedule on an auxiliary ray digraph, and lift the schedule through fresh clean linkages. For the matching lower bound, we use alternating orientations of products of a ray with a three-armed tree and prove a no-passing property for disjoint directed paths. 

We also determine the dual extremal parameter. Let $W(d)$ denote the largest width that is always forced by an end of in-degree $d$, with all branch rays remaining in that end, then
$$W(d)=\left\lfloor\frac{2d}{3}\right\rfloor+1
\qquad\text{for every }d\geq4,$$
with $W(1)=1$ and $W(2)=W(3)=2$. 
\end{abstract}

\section{Introduction}

We use standard terminology for graphs and digraphs~\cite{BangJensenGutin,Diestel}. A \emph{ray} is a one-way infinite path. In an undirected graph $G$, two rays are equivalent if, for every finite set $S\subseteq V(G)$, tails of both rays lie in the same component of $G-S$. An \emph{end} of $G$ is an equivalence class of rays under this relation. A subdivision of a graph or digraph is obtained by replacing each edge or arc by a path or dipath, respectively, so that the replacement paths have pairwise disjoint interiors and no internal vertex is a branch vertex.

Halin's grid theorem is one of the fundamental links between ends and grid structure in infinite graphs~\cite{Halin}. It states that an undirected graph containing infinitely many pairwise disjoint equivalent rays contains a subdivision of the hexagonal half-grid whose vertical rays lie in the same end. Kurkofka, Melcher and Pitz~\cite{KurkofkaMelcherPitz} later strengthened this result by showing that the vertical rays may be chosen from any prescribed infinite family of pairwise disjoint equivalent rays. These results concern ends of infinite degree and produce a grid of infinite width.

For ends of finite degrees, a natural problem is to determine the minimum degree of the end that guarantees a subdivision of a hexagonal grid of a prescribed finite width. In the undirected setting, Stein~\cite{Stein} obtained a sharp end-degree bound for finite-width $[k]\times\mathbb{N}$-grid minors. His extremal examples are products of a subdivided three-armed tree with a ray. The vertical tracks determine the degree of the end, while the three arms restrict the width of a grid that can occur. These examples give the factor $3/2$ that also appears in our lower bound below.

In a digraph, a ray is a one-way infinite directed path, while an anti-ray is a one-way infinite path directed towards its initial vertex. Zuther~\cite{ZutherThesis,ZutherEnds} introduced an end concept as follows. For rays or anti-rays $Q$ and $R$ in a digraph $D$, write $Q\leq R$ if there are infinitely many pairwise vertex-disjoint $Q$--$R$ dipaths, and write $Q\sim R$ if both $Q\leq R$ and $R\leq Q$. The equivalence classes under $\sim$ are called the ends of $D$. The \emph{in-degree} of an end containing rays is the maximum size of a family of pairwise vertex-disjoint rays in it.  Hamann and Heuer~\cite{HamannHeuerEndDegree} showed that this maximum is attained, including in the infinite case. For countable ends, they also characterized the combined in- and out-degrees in terms of end-exhausting sequences. They also proved that an end containing infinitely many pairwise disjoint rays contains a subdivision of the bidirected quarter-grid~\cite{HamannHeuerGrids}. Reich~\cite{Reich} independently proved a directed version of Halin's grid theorem by analysing an infinite auxiliary digraph encoding the connections among a prescribed family of pairwise disjoint equivalent rays. These results show that a directed end of infinite in-degree contains a standard infinite grid structure.

We next introduce the finite-width grid notations. For $n\ge1$, let
\[
L_i=x^i_0x^i_1x^i_2\cdots\qquad(1\le i\le n)
\]
be pairwise disjoint rays. The directed hexagonal grid of width $n$, denoted by $H_n$, is obtained by adding, for $1\le i<n$, the arcs
\[
\begin{array}{ll}
x^i_jx^{i+1}_j & \text{if $i$ is odd and $j\equiv1\pmod4$},\\
x^{i+1}_jx^i_j & \text{if $i$ is odd and $j\equiv3\pmod4$},\\
x^i_jx^{i+1}_j & \text{if $i$ is even and $j\equiv2\pmod4$},\\
x^{i+1}_jx^i_j & \text{if $i$ is even and $j\equiv0\pmod4$}.
\end{array}
\]
Deleting a common finite initial part from all rails, and then reindexing the remaining vertices, yields a tail isomorphic to $H_n$. Figure~\ref{fig:H4} illustrates a finite initial part of $H_4$. 
 
\begin{figure}[H]
\centering
\includegraphics[width=0.20\textwidth]{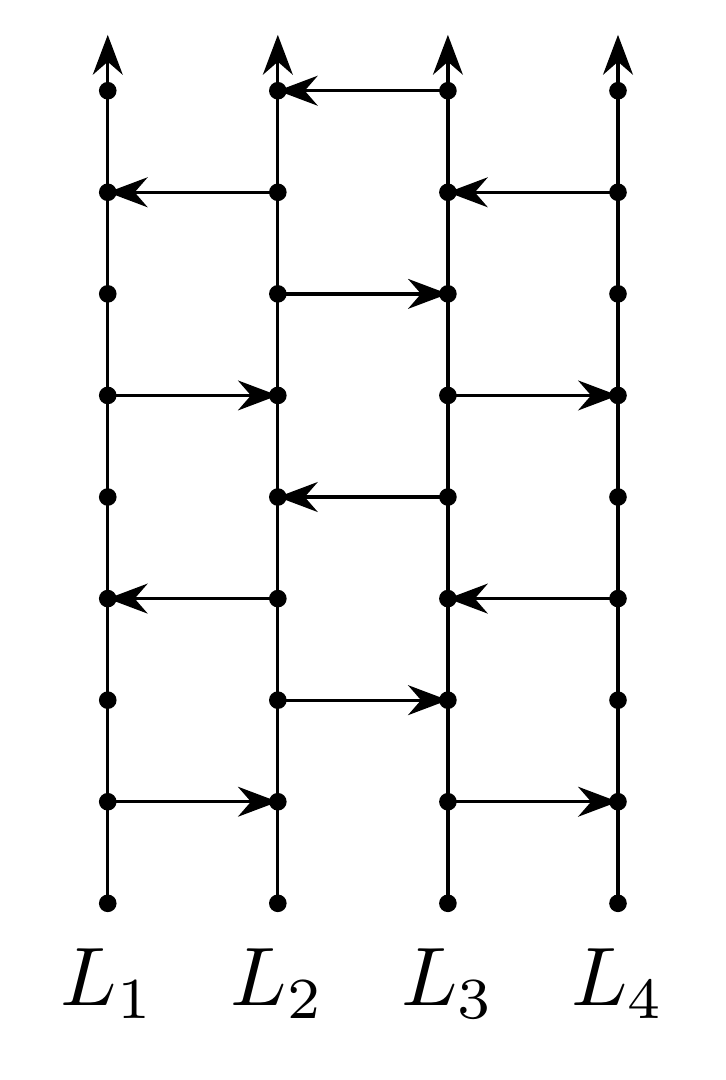}
\caption{A finite initial part of $H_4$.}
\label{fig:H4}
\end{figure}

 For every fixed width $n$, let $k(n)$ denote the least positive integer such that every digraph with an end of in-degree at least $k(n)$ contains a subdivision of the directed hexagonal grid of width $n$. Hamann and Heuer~\cite{HamannHeuerGrids} proved that $k(n)$ exists and provided upper and lower bounds on it. Their upper bound is obtained from a structural trichotomy for finite strong digraphs, while their lower bound is obtained from an orientation of Stein's three-armed product. They also posed the following open problem.

\begin{problem} ~\cite{HamannHeuerGrids} \label{problem:HH}
Let  $n$ be a positive integer. Determine the exact value of $k(n)$. 
\end{problem}

We give a complete solution to this problem as follows.

\begin{theorem}\label{thm:main}
Let $n\ge4$. We have
\[
k(n)=\left\lfloor\frac{3n}{2}\right\rfloor-1.
\]
In addition, $k(1)=1, k(2)=2, k(3)=4$.
\end{theorem}

We introduce the auxiliary ray digraph. Let $\cR$ be a finite family of pairwise disjoint rays. The \emph{auxiliary ray digraph} $A(\cR)$ has vertex set $\cR$. It contains the arc $RS$ exactly when there are infinitely many pairwise vertex-disjoint $R$--$S$ dipaths whose interiors avoid every ray of $\cR$. We call these dipaths \emph{clean $R$--$S$ linkages}. As observed in~\cite{HamannHeuerGrids}, if all rays in $\mathcal{R}$ belong to the same end, then $A(\mathcal{R})$ is strong.

For the upper bound, we choose a suitable finite family of pairwise disjoint rays in the end and form its auxiliary ray digraph. A finite labelled token-scheduling argument produces a closed directed schedule with the required ordered events. We then lift this schedule, using fresh clean linkages beyond successive carrier frontiers, to obtain a subdivision of $H_n$.

For the lower bound, we use directed products of a three-armed tree with a ray. Vertical cuts determine the in-degree of the unique end, while a no-passing property for disjoint directed paths excludes a subdivision of the required width. Suitable choices of the three arm lengths give the matching lower bound for every $n\geq 4$. The case $n=3$ is obtained from the separate three-ray construction $\Gamma_3$, and the cases $n=1,2$ follow directly.

Section 2 introduces the finite labelled-token framework and establishes the finite scheduling results needed later. Section 3 applies these results to directed graphs and proves the upper bound in Theorem~\ref{thm:main}. Section 4 constructs the extremal examples and proves the corresponding lower bound, completing the proof of Theorem~\ref{thm:main}. Section 5 reformulates the main result in terms of end in-degree and determines the largest directed hexagonal-grid width forced by a prescribed end in-degree.

\section{A finite vacancy theorem}

Let $G$ be a finite graph. A \emph{placement} of $r$ labelled tokens on $G$ assigns the labelled tokens injectively to vertices of $G$. The $h=|V(G)|-r$ unoccupied vertices are called \emph{holes}. A \emph{token slide} moves one labelled token along an edge into an adjacent hole. A finite sequence of slides is a \emph{closed schedule} if every labelled token returns to its initial vertex. Two labelled tokens are in \emph{contact} at a recorded placement when they occupy adjacent vertices. The \emph{contact graph} has the labels as vertices and the recorded contacts as edges.

\begin{lemma}\label{lem:occupation}
Let $G$ be a connected finite graph, and let $\pi$ be a placement on $G$ containing at least one labelled token. Any vertex of $G$ can be occupied at some placement along a finite sequence of token slides that starts and ends at $\pi$.
\end{lemma}

\begin{proof}
Let $v$ be any vertex of $G$. If $v$ is already occupied in $\pi$, there is nothing to prove. Suppose $v$ is a hole. Choose an occupied vertex $u$ of minimum distance from $v$, and let $v=v_0v_1\cdots v_k=u$ be a shortest $v$--$u$ path. By the choice of $u$, the vertices $v_0,\ldots,v_{k-1}$ are holes. Let $\tau$ be the labelled token occupying $u=v_k$. Slide $\tau$ successively as
\[
v_k\to v_{k-1},\quad
v_{k-1}\to v_{k-2},\quad
\ldots,\quad
v_1\to v_0.
\]
After these slides, $\tau$ occupies $v_0=v$. Reversing the sequence,
\[
v_0\to v_1,\quad
v_1\to v_2,\quad
\ldots,\quad
v_{k-1}\to v_k,
\]
returns $\tau$ to $u$.
\end{proof}

A tree is called a \emph{caterpillar} if its non-leaf vertices induce a path, possibly a trivial one or the empty path. Such a path is called a \emph{spine} of the caterpillar.

\begin{lemma}\label{lem:caterpillar}
Let $C$ be a caterpillar with $m$ edges. There is a placement of $m$ labelled tokens and exactly one hole on $C$, for which there exists a finite closed schedule whose contact graph contains a Hamilton path.
\end{lemma}

\begin{proof}
Since $C$ is a caterpillar, it is a tree. Choose a vertex $x$ as the root and place the unique hole at $x$. For every edge $e=uv$ of $C$, exactly one of $u$ and $v$ is closer to $x$. We call this vertex the \emph{parent end} of $e$ and the other vertex the \emph{child end}. For each edge $e$, place a labelled token $\tau_e$ at the child end.  Denote this rooted placement by $\pi$. We let $K$ be the component containing $\pi$ in the placement graph.

For a vertex $w\neq x$, the unique edge whose child end is $w$ is called the \emph{parent edge} of $w$, while the edges whose parent end is $w$ are called its \emph{child edges}.

Let $L(C)$ be the line graph of $C$. Let $e$ and $f$ be incident edges in $C$ with a common endpoint $w$. If one is the parent edge of $w$ and the other is a child edge, without loss of generality, let $e$ be the parent edge and $f$ a child edge of $w$. Then $\tau_e$ occupies $w$ and $\tau_f$ occupies the child end of $f$, so $\tau_e$ and $\tau_f$ are already adjacent in $\pi$. Now assume that both $e$ and $f$ are child edges of $w$. We may move the hole along the unique $x$--$w$ path and slide $\tau_e$ into $w$. Then two labelled tokens are adjacent in this placement. Thus every edge of $L(C)$ is realised as a contact at some placement in $K$.

Let $P=v_1v_2\cdots v_s$ be a spine of $C$. If $s=1$, all edges of $C$ are incident with $v_1$, so they may be listed in any order. Assume $s\geq2$. Traverse $P$ from $v_1$ to $v_s$. At $v_1$, list its pendant edges before the spine edge $v_1v_2$. At each internal spine vertex $v_i$, list its pendant edges between the incoming edge $v_{i-1}v_i$ and the outgoing edge $v_iv_{i+1}$. At $v_s$, list its pendant edges after $v_{s-1}v_s$. In this way, every edge of $C$ is listed exactly once and any two consecutive edges in the resulting order are incident. Write this ordering as $e_1,e_2,\ldots,e_m.$ Since every two consecutive edges $e_i$ and $e_{i+1}$ are incident in $C$, they are adjacent as vertices of $L(C)$ for every $1\leq i<m$. Moreover, $e_1,\ldots,e_m$ are precisely the $m$ edges of $C$, and hence the $m$ vertices of $L(C)$. Hence $e_1,e_2,\ldots,e_m$ is a Hamilton path in $L(C)$.

For each $i<m$, there exists a placement in $K$ witnessing the contact between $\tau_{e_i}$ and $\tau_{e_{i+1}}$. Since $K$ is finite and connected, we may reach each chosen witness from $\pi$ and then return to $\pi$. Concatenating these out-and-back walks gives a finite closed schedule whose contact graph contains $\tau_{e_1}\tau_{e_2}\cdots\tau_{e_m}$.
\end{proof}

\begin{theorem}\label{thm:tree-vacancy}
Let $T$ be a tree on $n\ge1$ vertices. There exists a placement of $r(n)$ labelled tokens and $h(n)$ holes on $T$ for which there is a finite closed schedule whose contact graph contains a Hamilton path.
\end{theorem}

\begin{proof}
Let $h(n)=\lfloor(n-1)/3\rfloor, r(n)=n-h(n).$ We proceed by induction on $n$. If $T$ is a path, place the $r(n)$ labelled tokens on consecutive vertices. Their static contacts already form a Hamilton path, so the empty schedule suffices.

Suppose that $T$ is not a path. Then choose three distinct leaves $\ell_1,\ell_2,\ell_3$, let
\[
T_0=T-\{\ell_1,\ell_2,\ell_3\}.
\]
Then $T_0$ is a tree. For $j=1,2,3$, let $a_j$ be the neighbour of $\ell_j$. Since
\[
h(n-3)=h(n)-1,
\qquad
r(n-3)=r(n)-2,
\]
the induction hypothesis gives a placement on $T_0$ with $h(n)-1$ holes and $r(n)-2$ labelled tokens, together with a finite closed schedule $W$ whose contact graph contains a Hamilton path. Let $p=r(n)-2$. Relabel these labelled tokens as $Q_1,\ldots,Q_p$ in the order of this Hamilton path.

By inserting the excursions from  Lemma~\ref{lem:occupation}, for each $j$, we may assume that some labelled token $Q_{\alpha_j}$ occupies $a_j$ at a recorded point of every period of $W$. Thus for $1\le i<p$, $1\le j\le3$, we have
\begin{equation}\label{eq:lane-pairs}
Q_iQ_{i+1},
\qquad
Q_{\alpha_j}\ell_j
\end{equation}
are adjacent at recurring placements of $W$. Here $Q_i$ denotes the moving lane traced by the labelled token $Q_i$ while $W$ is performed, whereas each $\ell_j$ is a fixed lane.

Construct a caterpillar $C$ whose spine is $Q_1Q_2\cdots Q_p$ and whose leaf $\ell_j$ is attached to $Q_{\alpha_j}$. See Figure~\ref{fig:three-leaf}. Its $p+3$ vertices represent the $p$ dynamic lanes $Q_1,\ldots,Q_p$ and the three fixed lanes $\ell_1,\ell_2,\ell_3$, while its $p+2=r(n)$ edges index the labelled meta-tokens.
By Lemma~\ref{lem:caterpillar}, the one-hole puzzle on $C$ has a rooted placement of $r(n)$ labelled meta-tokens and a finite closed meta-schedule $M$ whose contact graph contains a Hamilton path.

We now simulate $M$ on the original tree $T$. Figure~\ref{fig:flattening} illustrates the correspondence between the meta-level and the physical tree. During the simulation we keep the following conditions:
\begin{enumerate}[label=(I\arabic*),leftmargin=3.5em]
\item each lane of $C$ carries either one labelled meta-token or the unique meta-hole;
\item a lane carrying a labelled meta-token contains one labelled physical token with the same label;
\item the meta-hole lane and the $h-1$ background holes are empty in $T$;
\item each dynamic lane $Q_i$ is represented at the vertex prescribed by the current state of $W$.
\end{enumerate}

At the initial state of $W$, place the labelled physical token corresponding to each labelled meta-token at the vertex represented by its lane, and leave the lane carrying the unique meta-hole empty. The remaining $h-1$ holes of the initial placement on $T_0$ are designated as the background holes. Thus $(I1)-(I4)$ hold initially.

Consider one reference slide of $W$, say $Q_i:u\to v$, where $v$ is a background hole. If the $Q_i$-lane carries a labelled meta-token, perform the same physical slide and declare $u$ to be the new background hole. If the $Q_i$-lane is the meta-hole, no labelled physical token moves. The meta-hole follows the lane to $v$, and $u$ becomes the new background hole. In either case the invariant is preserved.

To simulate a meta-slide across an edge of $C$, continue $W$ until the corresponding pair in~\eqref{eq:lane-pairs} is adjacent in $T$, and slide the labelled physical token into the meta-hole. Similarly, when $M$ records a contact between two labelled meta-tokens, continue $W$ until the edge of $C$ occupied by these labelled tokens is realised by two adjacent physical lanes, and record the corresponding physical contact. Since every pair in~\eqref{eq:lane-pairs} is witnessed at a prescribed placement of $W$, periodic repetition of the closed schedule $W$ makes every such adjacency recur, so each meta-operation can be realised after finitely many reference steps.

After the last meta-step, complete the current period of $W$. The dynamic lanes, all meta-labels, and the background holes return to their initial positions. Hence the physical schedule is closed. It has $(h-1)+1=h$ holes and $p+2=r(n)$ labelled tokens, and its contact graph contains the Hamilton path supplied by Lemma~\ref{lem:caterpillar}.
\end{proof}

\begin{figure}[H]
\centering
\includegraphics[width=0.62\textwidth]{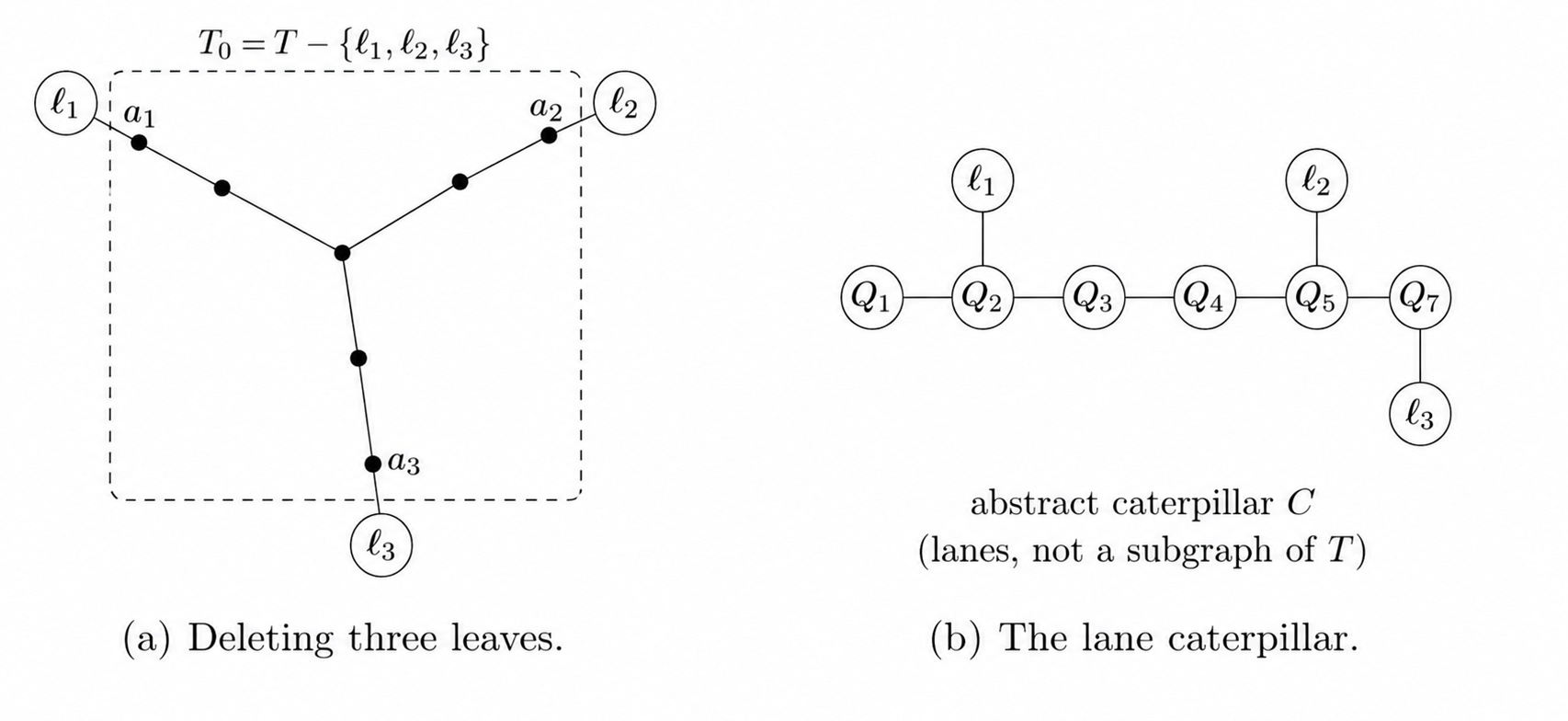}
\caption{The three-leaf reduction in the proof of Theorem~\ref{thm:tree-vacancy}.}
\label{fig:three-leaf}
\end{figure}

\begin{figure}[H]
\centering
\includegraphics[width=0.60\textwidth]{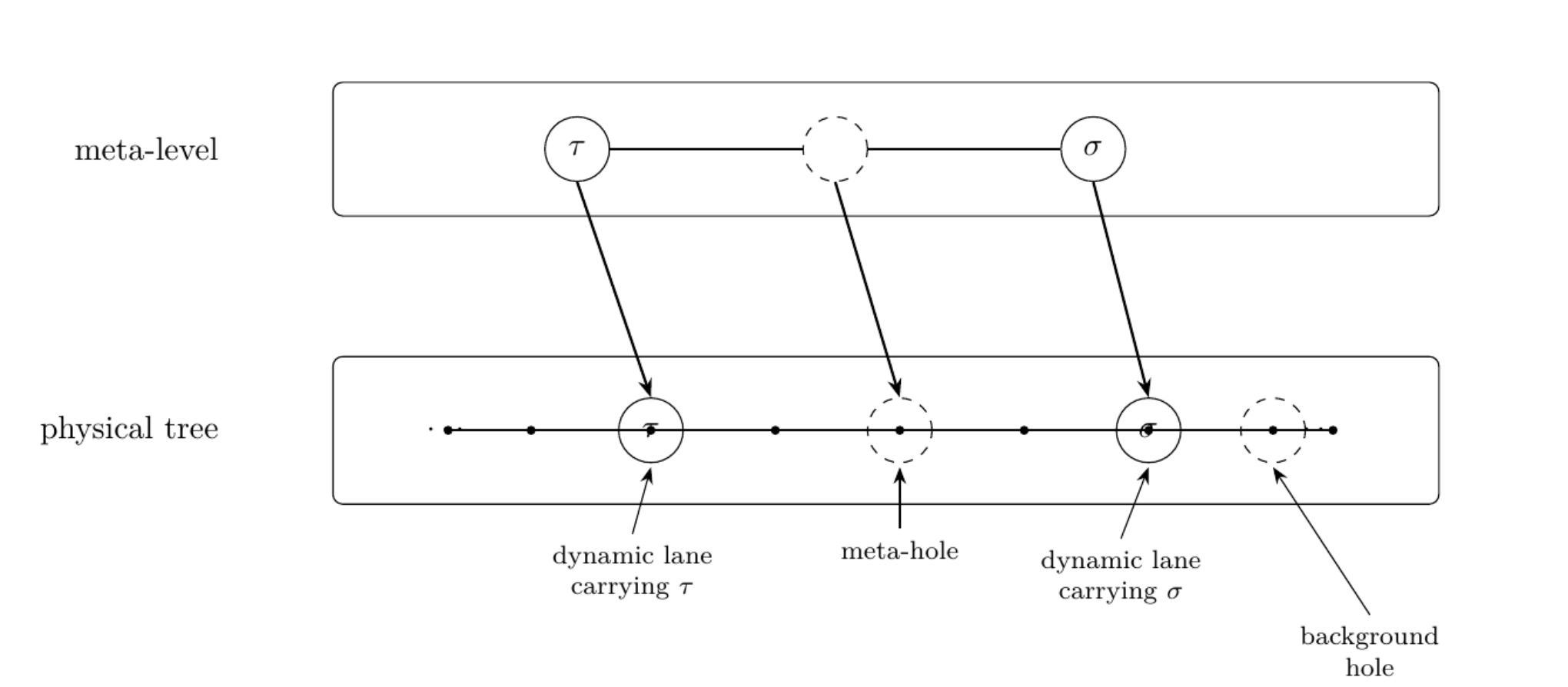}
\caption{The correspondence between the meta-schedule and the physical tree.}
\label{fig:flattening}
\end{figure}

\begin{corollary}\label{cor:connected-vacancy}
Let $G$ be a connected graph on $n$ vertices. There is a placement of $n-\lfloor(n-1)/3\rfloor$ labelled tokens and $\lfloor(n-1)/3\rfloor$ holes on $G$ for which there exists a finite closed schedule whose contact graph contains a Hamilton path.
\end{corollary}

\begin{proof}
Apply Theorem~\ref{thm:tree-vacancy} to a spanning tree of $G$. Every slide and every contact used in that schedule is also available in $G$.
\end{proof}

Let $A$ be a digraph, and $U(A)$ denote the underlying graph in which $uv$ is an edge whenever at least one of $uv$ and $vu$ is an arc of $A$. A slide from $u$ to a hole at $v$ is legal only if $uv\in E(A)$.

The \emph{ordinary placement graph} has the labelled placements as its vertices and two placements are adjacent when one is obtained from the other by a token slide along an edge of $U(A)$. The \emph{directed placement graph} has the same vertices, and its arcs are the legal directed slides in $A$.

If the labelled token $x$ occupies $u$, the labelled token $y$ occupies $v$, and $uv\in E(A)$, we record the \emph{direct event} $x\xRightarrow{1}y.$ A \emph{handover event} $x\xRightarrow{2}y$ occurs when $x$ leaves a carrier and the immediately following slide places $y$ on the same carrier, see Figure~\ref{fig:directed-moves}(b). We write $x\Rightarrow y$ if either $x\xRightarrow{1}y$ or $x\xRightarrow{2}y$ occurs. In the lifting argument, the carrier segment between the two
switches will be used as a directed rung. See Figure~\ref{fig:directed-moves}(b).

\begin{lemma}\label{lem:placement-scc}
Let $A$ be a finite strong digraph, and let $1\leq r<|V(A)|$. For a fixed set of $r$ labelled tokens, the strongly connected components of the directed placement graph are exactly the connected components of the ordinary placement graph on $U(A)$.
\end{lemma}

\begin{proof}
It is enough to prove that every legal directed slide can be reversed by a directed walk. Suppose a labelled token slides along $uv\in E(A)$ into a hole at $v$. Since $A$ is strong, $uv$ lies on a directed cycle $Z$. We keep all labelled tokens outside $Z$ fixed and use only slides along $Z$.

Slides on a directed cycle preserve the cyclic order of the labels. Conversely, we show that any two placements with the same cyclic order are mutually reachable. Write the labels in cyclic order as $z_1,\ldots,z_t$, and let $g_i$ be the number of holes after $z_i$ and before $z_{i+1}$, with indices read cyclically. Whenever $g_i>0$, sliding $z_i$ changes the gap vector by
\begin{equation}\label{eq:gap-transfer}
g_i\mapsto g_i-1,
\qquad
g_{i-1}\mapsto g_{i-1}+1.
\end{equation}
First, repeated use of this move transports vacancy units around the cycle, so any gap vector with the same total number of holes can be reached.
Second, since there is at least one hole, choose $i$ with $g_i>0$. Slide the labelled tokens in the order
\[
z_i,z_{i-1},\ldots,z_{i+1},
\]
with indices read cyclically. Each slide is enabled by the vacancy created by the preceding slide. At the end, the gap vector is restored and every labelled token has advanced by one vertex of $Z$. Hence any two
placements on $Z$ with the same cyclic order are mutually reachable. In particular, the placements immediately before and after the slide $u\to v$ lie in the same directed strong component on $Z$.

Now consider an edge of the ordinary placement graph. At least one orientation of the corresponding carrier edge belongs to $A$, so one of the two slides is legal. The preceding argument makes that slide reversible. Thus every ordinary edge lies within a directed strong component. Hence each ordinary connected component is contained in one directed strong component. The reverse containment is immediate, because every directed slide is also an ordinary slide.
\end{proof}

\begin{lemma}\label{lem:contact-symmetry}
Let $A$ be a finite strong digraph and let $\pi$ be a placement with at least one hole. Suppose that the labelled tokens $x$ and $y$ occupy $u$ and $v$ respectively, and $uv\in E(A)$. Then both ordered events
\[
x\Rightarrow y
\qquad\text{and}\qquad
y\Rightarrow x
\]
can be witnessed within the strongly connected component of the directed placement graph containing $\pi$.
\end{lemma}

\begin{proof}
The arc $uv$ already gives the direct event $x\xRightarrow{1} y$. If $vu\in E(A)$, then $y\xRightarrow{1}x$ is also direct, so assume that $vu\notin E(A)$.

\emph{Case 1: a hole is reachable from $v$ in $A-u$.}
Choose a shortest dipath
\[
v=p_0\to p_1\to\cdots\to p_t,
\]
where $p_t$ is a hole. By minimality, $p_1,\ldots,p_{t-1}$ are occupied. Move the vacancy backwards along
this path until $p_1$ becomes empty, and then perform
\[
y:v\to p_1,
\qquad
x:u\to v.
\]
They give a handover event $y\xRightarrow{2}x$  on the carrier $v$.

\emph{Case 2: no hole is reachable from $v$ in $A-u$.}
Let $S=\Reach_{A-u}(v)$. Every vertex of $S$ is
occupied. Choose a dipath from $v$ to $u$ that first reaches $u$ at its final vertex,
\[
v=c_1\to c_2\to\cdots\to c_s\to u.
\]
Together with $u\to v$, it gives the directed cycle
\begin{equation}\label{eq:cycle-Z}
Z=u\to c_1\to\cdots\to c_s\to u.
\end{equation}
Choose a shortest dipath from $u$ to a hole,
\[
u=q_0\to q_1\to\cdots\to q_t.
\]
Since $S$ is forward-closed in $A-u$, if a shortest $u$ to a hole dipath enters $S$, it can leave $S$ only through $u$. This would force the path to revisit $u$, contradicting simplicity. Hence the vertices $q_1,\ldots,q_t$ lie outside $S$. Since $c_1,...,c_s$ are contained in $S$, this $u$ to a hole dipath meets the directed cycle $Z$ only at $u$. Move the vacancy backwards along this path and then slide $x$ from $u$ to $q_1$. The vertex $u$ is now empty.

Keep the hole at $u$ while ticking the occupied vertices of $Z$. One tick consists of the slides
\[
c_s\to u,
\quad c_{s-1}\to c_s,
\quad\ldots,
\quad c_1\to c_2,
\quad u\to c_1.
\]
At the end of the tick the hole is again at $u$. After $s-1$ ticks, the labelled token $y$ is at $c_s$. Slide it to $u$.
The labelled token $x$ is still at $q_1$, so the arc $u\to q_1$ witnesses the direct event $y\xRightarrow{1}x$.

All placements used above lie in the ordinary placement component of $\pi$. By Lemma~\ref{lem:placement-scc}, they therefore lie in the same strongly connected component.
\end{proof}

\begin{figure}[H]
\centering
\includegraphics[width=0.95\textwidth]{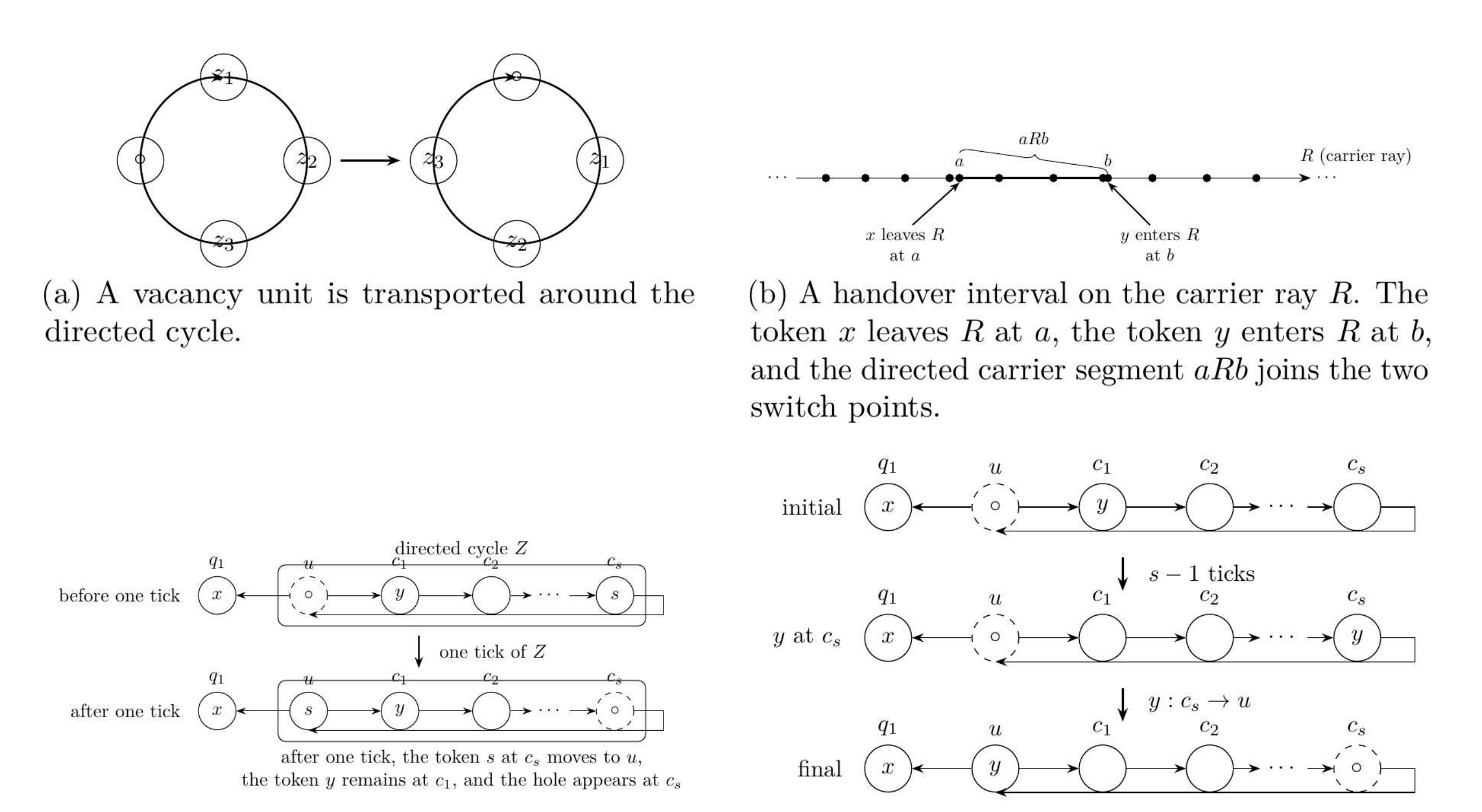}
\caption{Directed moves used in Lemmas~\ref{lem:placement-scc} and~\ref{lem:contact-symmetry}. Dashed circles denote holes.}
\label{fig:directed-moves}
\end{figure}

\begin{theorem}\label{thm:directed-augmentation}
Let $A$ be a strong digraph on $n\ge4$ vertices. There exists a placement of $r$ labelled tokens and $h$ holes on $A$ for which there is a finite closed directed schedule whose ordered event graph contains the bidirected Hamilton path
\[
z_1\leftrightarrow z_2\leftrightarrow\cdots\leftrightarrow z_r.
\]
Moreover, each arrow is witnessed by a direct event or a handover event.
\end{theorem}

\begin{proof}
Let $T$ be a spanning tree of $U(A)$, and let $h(n)=\lfloor(n-1)/3\rfloor,r=n-h$. By Theorem~\ref{thm:tree-vacancy}, one ordinary placement component contains contacts $z_iz_{i+1}$ for all $1\le i<r$. Lemma~\ref{lem:placement-scc} identifies this component with a strongly connected component of the directed placement graph. Within that component, Lemma~\ref{lem:contact-symmetry} gives witnesses for both $z_i\Rightarrow z_{i+1}$ and $z_{i+1}\Rightarrow z_i$.

There are only finitely many required witnesses. Join them in any chosen order by directed walks inside the strong component, and after the last witness return to the initial placement. The resulting directed schedule is closed and has the required bidirected Hamilton path in its ordered event graph.
\end{proof}

\section{From a finite schedule to an infinite grid}

We first record the finite-avoidance fact used throughout the lifting argument.

\begin{lemma}\label{lem:fresh}
Let $\cR$ be a finite family of pairwise disjoint carrier rays in a digraph $D$, and let $A=A(\cR)$. Suppose that $RS\in E(A)$. For every finite set $F\subseteq V(D)$ and vertices $r\in R$ and $s\in S$, there exists a clean $R$--$S$ linkage disjoint from $F$ whose initial vertex on $R$ lies strictly after $r$ and whose terminal vertex on $S$ lies strictly after $s$.
\end{lemma}

\begin{proof}
Let $R_{\le r}$ and $S_{\le s}$ denote the initial segments of $R$ and $S$ ending at $r$ and $s$, and put
\[
F'=F\cup V(R_{\le r})\cup V(S_{\le s}).
\]
The set $F'$ is finite. Since $RS\in E(A)$, there is an infinite pairwise vertex-disjoint family of clean $R$--$S$ linkages. Only finitely many members of this family meet $F'$, so one of them is disjoint from $F'$. Its endpoints therefore lie strictly after $r$ on $R$ and $s$ on $S$, as required.
\end{proof}

\begin{lemma}\label{lem:lifting}
Let $\cR$ be a finite pairwise disjoint family of carrier rays in a digraph $D$, and let $A=A(\cR)$. Suppose that there is a placement of $r\geq 2$ labelled tokens and at least one hole on $A$ for which there exists a finite closed directed schedule whose ordered event graph contains
\[
z_1\leftrightarrow\cdots\leftrightarrow z_r.
\]
Then $D$ contains a subdivision of $H_r$. Moreover, if all rays in $\mathcal R$ belong to the same end $\omega$, then the subdivision may be chosen so that all its branch rays belong to $\omega$.
\end{lemma}

\begin{proof}
By repeating the schedule finitely many times and marking suitable occurrences of the required events, we may arrange the required events in the following four phases:
\begin{equation}\label{eq:four-phases}
\begin{aligned}
(1)&\quad z_i\Rightarrow z_{i+1} &&(i\text{ odd}),\\
(2)&\quad z_i\Rightarrow z_{i+1} &&(i\text{ even}),\\
(3)&\quad z_{i+1}\Rightarrow z_i &&(i\text{ odd}),\\
(4)&\quad z_{i+1}\Rightarrow z_i &&(i\text{ even}).
\end{aligned}
\end{equation}
The pairs in each phase form a matching. After the final marked event of phase (4), complete the current copy of the closed schedule. Thus every labelled token is returned to its initial carrier. We repeat this four-phase word indefinitely.

We realise the resulting infinite schedule recursively in $D$. For each carrier $R$, let its \emph{frontier} be the last vertex used so far by a branch segment, a switching linkage, or a rung. At a finite stage, let $F$ consist of all vertices already used, together with the initial carrier segments ending at the current frontiers. Since only finitely many operations have been performed up to this stage and each operation uses only finitely many vertices, the set of previously used vertices is finite. Moreover, $R$ is finite and every initial carrier segment ending at a current frontier is finite. Hence $F$ is finite.

Whenever the schedule uses an auxiliary arc $RS$, apply Lemma~\ref{lem:fresh} to choose a fresh clean $R$--$S$ linkage disjoint from $F$ and with both endpoints beyond the current frontiers. There are three types of operations. See Figure ~\ref{fig:lifting-operations}.

\begin{enumerate}[label=(L\arabic*),leftmargin=3.2em]
\item Suppose that a labelled token slides from $R$ to the vacant carrier $S$. Choose a fresh clean $R$--$S$ linkage. Extend the labelled trajectory along $R$ to the initial vertex of this linkage, traverse the linkage, and end at its terminal vertex on $S$. Update the frontiers of $R$ and $S$.

\item Suppose that a direct event $x\xRightarrow{1}y$ occurs on occupied carriers $R$ and $S$. Choose a fresh clean $R$--$S$ linkage. Extend the trajectories of $x$ and $y$ along their current carriers to the two endpoints of the linkage, and mark the linkage as a rung. The abstract placement is unchanged.

\item Suppose that a handover event $x\xRightarrow{2}y$ on a carrier $R$. The schedule first moves $x$ off $R$ and then moves $y$ onto $R$. Realise these two slides successively as in (L1). Let $a$ be the point at which $x$ leaves $R$, and let $b$ be the point at which $y$ enters $R$. When the second switching linkage is chosen, choose its endpoint $b$ on $R$ beyond the current frontier. In particular, $b$ lies strictly after $a$. The directed segment $aRb$ is then a rung from the branch ray of $x$ to the branch ray of $y$. Mark this segment as used and move the frontier of $R$ past $b$.
\end{enumerate}

For (L1) and (L2), each new linkage is chosen disjoint from all previously used vertices, with its endpoints beyond the current frontiers. Thus in (L1) the extended trajectory meets no earlier labelled trajectory, while in (L2) the interior of the new rung meets neither a labelled trajectory nor an earlier rung. In (L3), the carrier $R$ is vacant between the departure point $a$ of $x$ and the arrival point $b$ of $y$. Moreover, $b$ is chosen beyond the frontier reached after $x$ leaves $R$, so the interior of $aRb$ contains no previously used vertex. Hence it meets neither a labelled trajectory nor an earlier rung. It follows inductively that the labelled trajectories are pairwise disjoint and that the interiors of the marked rungs are disjoint from all labelled trajectories and from one another.

After each phase in ~\eqref{eq:four-phases}, extend every trajectory not incident with that phase's matching along its current carrier to a new vertex beyond the current frontier, and declare this vertex to be a degree-two marker. Since the pairs in each phase form a matching, every labelled trajectory receives exactly one distinguished vertex in that phase: either a rung endpoint or a degree-two marker. Hence, after every phase, each labelled trajectory extends strictly beyond its previous frontier. Since the four-phase word is repeated indefinitely, every labelled trajectory contains infinitely many vertices and is therefore a directed ray.

At the end of each period, every labelled token returns to its initial carrier. Fix a labelled branch ray $B$, and let $R$ be its initial carrier. Since the construction moves beyond the current frontiers after every phase, there are vertices $u_1,u_2,u_3,\ldots\in V(B)\cap V(R)$ that occur successively farther out on both $B$ and $R$.
For $j\geq 1$, the segments $R[u_{2j-1},u_{2j}]$ form an infinite family of pairwise vertex-disjoint $R$--$B$ dipaths. Hence $R\leq B$. Similarly, the segments $B[u_{2j},u_{2j+1}]$ form an infinite family of pairwise vertex-disjoint $B$--$R$ dipaths, and hence $B\leq R$. Therefore $B\sim R$. Thus, if the carriers belong to a common end $\omega$, every labelled branch ray belongs to $\omega$.

Along each labelled trajectory, the rung endpoints and these dummy markers then occur in the four residue-class order used in the definition of $H_r$. The unmarked switching linkages and carrier segments form the longitudinal parts of the subdivided rails and may be suppressed as degree-two vertices. The marked rungs together with the $r$ labelled branch rays therefore form a subdivision of $H_r$.
\end{proof}

\begin{figure}[H]
\centering
\includegraphics[width=0.92\textwidth]{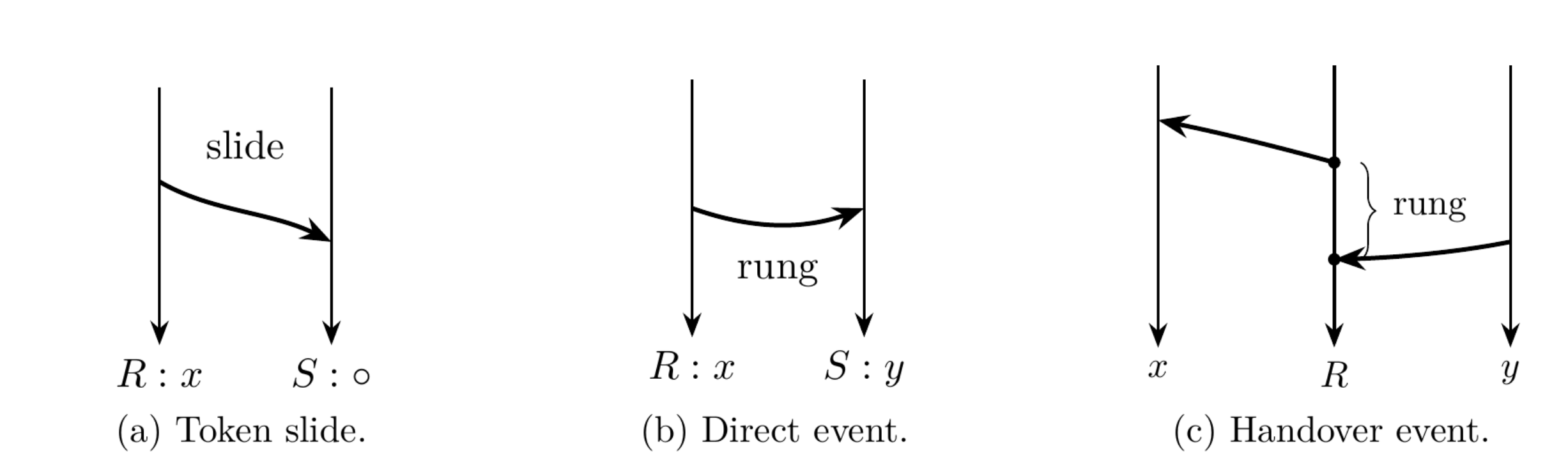}
\caption{The three operations in the proof of Lemma~\ref{lem:lifting}.}
\label{fig:lifting-operations}
\end{figure}

\begin{remark}
The lifting argument in Lemma~\ref{lem:lifting} is not specific to the four-phase pattern defining $H_r$. Let \(\mathcal{M}_0,\mathcal{M}_1,\ldots,\mathcal{M}_{p-1}\) be a fixed finite cyclic sequence of directed matchings on the same label set. Suppose that some finite repetition of the closed directed labelled token schedule contains witnesses for all ordered events prescribed by these matchings, with the matchings occurring in the stated cyclic order. Since each $\mathcal{M}_i$ is a matching, the witnesses within a fixed phase may be taken in any order.

Repeating this finite word and applying the same frontier construction produces fresh rungs in the prescribed cyclic order. After inserting degree-two markers on the rails not incident with each matching, the resulting branch rays and rungs form a subdivision of the periodic directed strip determined by \(\mathcal{M}_0,\ldots,\mathcal{M}_{p-1}\). The grid $H_r$ is the special case given by the four matchings in (\ref{eq:four-phases}).
\end{remark}

The preceding results reduce the upper-bound problem to a finite scheduling problem on the auxiliary ray digraph. Theorem~\ref{thm:directed-augmentation} provides the required closed directed schedule, and Lemma~\ref{lem:lifting} lifts such a schedule to a subdivision of the directed hexagonal grid in the original digraph. We now use these results to establish the upper bound for $k(n)$.

\begin{theorem}\label{thm:upper}
For $n\ge4$,
\[
k(n)\le\left\lfloor\frac{3n}{2}\right\rfloor-1.
\]
In addition, $k(3)\le4$. Moreover, whenever an end $\omega$ satisfies the corresponding in-degree bound, the resulting subdivision may be chosen so that all its branch rays belong to $\omega$.
\end{theorem}

\begin{proof}
Choose $q$ pairwise disjoint rays in the end and form the auxiliary ray digraph $A$ on these $q$ rays. Since they belong to the same end, $A$ is strong. Theorem~\ref{thm:directed-augmentation} gives a closed schedule, and Lemma~\ref{lem:lifting} turns that schedule into a directed hexagonal grid.

For $n=3$, take $q=4$. Then $h=1$, $r=q-h=3$, so Lemma~\ref{lem:lifting} gives a subdivision of $H_3$.

Now let $n\ge4$. If $n=2m$, then $q=3m-1$, $h=m-1$, and hence $r=2m=n$. If $n=2m+1$, choose $q=3m$. In this case $h=m-1$, so $r=2m+1=n$. Thus the stated number of disjoint rays forces $H_n$ in both parity cases.

Since in each case the chosen carrier rays all belong to the original end $\omega$, Lemma~\ref{lem:lifting} shows that all branch rays of the resulting subdivision may also be chosen in $\omega$.
\end{proof}

\section{Sharp obstructions and proof of the main theorem}

\begin{proposition}\label{prop:small12}
We have $k(1)=1$ and $k(2)=2$.
\end{proposition}

\begin{proof}
Since $H_1$ is a ray, every end contains an $H_1$-subdivision. Hence $k(1)=1$.
For $n=2$, let $\omega$ be an end of in-degree at least two. Choose two disjoint rays $R,S\in\omega$. The auxiliary digraph $A(\{R,S\})$ is strong, so both $RS$ and $SR$ are arcs of $A(\{R,S\})$. By Lemma~\ref{lem:fresh}, we can choose clean linkages alternately in the two directions and successively farther out on $R$ and $S$. Suitable tails of $R$ and $S$ together with these linkages form a subdivision of $H_2$. Hence $k(2)\le2$. For the reverse inequality, let $D$ be a single directed ray. Its unique end has in-degree one, while $D$ contains no subdivision of $H_2$, since such a subdivision has two disjoint branch rays. Hence $k(2)\ge2$. Therefore $k(2)=2$.
\end{proof}

The case of width three is handled by the circular-grid example from~\cite{HamannHeuerGrids}. We give an explicit
description for later use. Let $\Gamma_3$ be the digraph with vertices $v_0,v_1,\ldots$ and arcs
\begin{equation}\label{eq:Gamma3}
v_pv_{p+1}\quad(p\ge0),
\qquad
v_pv_{p+5}\quad(p\ge0\text{ and }p\text{ is even}).
\end{equation}
Taking the arc $v_pv_{p+5}$ when $p$ is even and the arc $v_pv_{p+1}$ when $p$ is odd gives the three rays
\[
R_1=v_0v_5v_6v_{11}\cdots,
\qquad
R_2=v_1v_2v_7v_8\cdots,
\qquad
R_3=v_3v_4v_9v_{10}\cdots.
\]

\begin{lemma}\label{lem:Gamma3}
The digraph $\Gamma_3$ has an end of in-degree three, but contains no subdivision of $H_3$.
\end{lemma}

\begin{proof}
For $q\ge0$, write
\[
\begin{aligned}
v_{6q}&=r^1_{2q+1}, & v_{6q+1}&=r^2_{2q+1}, & v_{6q+2}&=r^2_{2q+2},\\
v_{6q+3}&=r^3_{2q+1}, & v_{6q+4}&=r^3_{2q+2}, & v_{6q+5}&=r^1_{2q+2},
\end{aligned}
\]
and set
\[
L_j=\{r^1_j,r^2_j,r^3_j\}.
\]
Each $L_j$ contains one vertex of each of $R_1,R_2,R_3$ and separates a finite initial segment of $\Gamma_3$
from all sufficiently late vertices. Hence every directed ray meets every sufficiently late $L_j$.

Let $Q_1,Q_2,Q_3$ be three pairwise disjoint rays. For every sufficiently large $j$, each $Q_i$ meets $L_j$, and the three intersections are distinct. Since $|L_j|=3$, $Q_1,Q_2,Q_3$ contains every vertex of every sufficiently late vertex of $L_j$, and therefore every sufficiently late vertex of $\Gamma_3$.

For sufficiently large $p$, let $x_p=1$ if the selected arc of $Q_1\cup Q_2\cup Q_3$ leaving $v_p$ is $v_pv_{p+5}$,
and let $x_p=0$ if it is $v_pv_{p+1}$. Set $x_p=0$ for odd $p$. Every sufficiently late vertex has selected in-degree one. At $v_q$ this gives
\[
(1-x_{q-1})+x_{q-5}=1,
\]
and hence
\begin{equation}
x_p=x_{p-4}
\end{equation}
for all sufficiently large $p$. Thus the tail pattern is determined by the two even residue classes modulo four. The corresponding successor map has the following numbers of cofinal orbits:
\[
\begin{array}{c|cccc}
(x_{0\pmod4},x_{2\pmod4})&(0,0)&(1,0)&(0,1)&(1,1)\\\hline
\text{number of cofinal orbits}&1&2&2&3.
\end{array}
\]
Indeed, tracing the successor map $p\mapsto p+1$ or $p\mapsto p+5$ in the four possible eventual residue patterns gives exactly the orbit counts displayed above. Since $Q_1,Q_2,Q_3$ are three cofinal orbits, the last case must occur. Hence $x_p=1$ for every sufficiently large even $p$. Up to permutation, the tails of $Q_1,Q_2,Q_3$ are the tails of $R_1,R_2,R_3$.

Suppose that $\Gamma_3$ contains a subdivision of $H_3$, and let $Q_1,Q_2,Q_3$ be its branch rays. By the preceding paragraph, their tails are the tails of $R_1,R_2,R_3$ and together contain every sufficiently late vertex of $\Gamma_3$. The vertex index strictly increases along every dipath in $\Gamma_3$. Choose a rung of a hypothetical $H_3$ subdivision whose endpoints lie beyond the finite initial part above. Every internal vertex of this rung is then also sufficiently late. Since every sufficiently late vertex already lies on one of the three branch rays and the interior of a rung is disjoint from all branch rays, every sufficiently late rung is a single arc.

The only late arcs between distinct rays have the cyclic orientations
\[
R_1\to R_2,
\qquad R_2\to R_3,
\qquad R_3\to R_1.
\]
Thus no pair of branch rays has sufficiently late rungs in both directions. This contradicts the definition of $H_3$. Therefore $\Gamma_3$ contains no subdivision of $H_3$.

Finally, for every ordered pair $R_i,R_j$, the periodic tail of $\Gamma_3$ contains directed $R_i$--$R_j$ linkages in infinitely many pairwise disjoint successive blocks, so $R_1,R_2,R_3$ lie in one end. Since every sufficiently late ray meets every sufficiently late three-vertex level $L_j$, and $|L_j|=3$, four pairwise disjoint rays are impossible. The end has in-degree exactly three.
\end{proof}

By Lemma~\ref{lem:Gamma3} and Theorem~\ref{thm:upper},
\begin{equation}\label{eq:k3}
k(3)=4.
\end{equation}

\medskip
\noindent\textbf{Construction 1}
For positive integers $a\ge b\ge c\ge1$, let $T(a,b,c)$ be the tree obtained by identifying one endpoint of three paths of lengths $a,b,c$ to a common centre $o$. Consider the Cartesian product
\[
Y(a,b,c)=T(a,b,c)\mathbin{\square}\N,
\]
where the second coordinate will be regarded as height.

Direct every vertical edge towards increasing height. In the horizontal layer at height $t$, direct each arm away from $o$ when $t$ is even and towards $o$ when $t$ is odd. We write $\widetilde{Y}(a,b,c)$ for the resulting digraph. See Figure~\ref{fig:three-arm}(a). For $t\in\N$, let $C_t=\{(w,t)(w,t+1):w\in V(T(a,b,c))\}$ be the vertical cut after layer $t$.

\begin{lemma}\label{lem:no-passing}
Let $\widetilde{Y}(a,b,c)$ and $C_t$ be as defined in Construction~1. Every directed ray in $\widetilde{Y}(a,b,c)$ uses exactly one arc of $C_t$ for all sufficiently large $t$. Hence pairwise disjoint directed rays occupy distinct tracks on every sufficiently large cut.

Fix an arm $A=oA_1\cdots A_d$ and a directed ray $S$ belonging to a pairwise vertex-disjoint family of directed rays. During any interval of cuts on which $S$ stays on tracks of $A\cup\{o\}$, no dipath vertex-disjoint from $S$ can cross from the $o$-side of $S$ to its outer side, or conversely.
\end{lemma}

\begin{proof}
The height coordinate is non-decreasing along every dipath. Since every finite set of layers contains only finitely many vertices, a directed ray has unbounded height. Hence the ray must use an arc of $C_t$ for every sufficiently large $t$. Once the ray crosses $C_t$, it cannot return to height $t$. Hence it uses exactly one arc of every sufficiently late $C_t$. Two disjoint rays cannot use the same vertical arc.

Number the tracks of $A\cup\{o\}$ by $0,1,\ldots,d$ with $0$ corresponding to $o$. Let $P$ be a dipath disjoint from $S$. While $P$ stays on tracks of $A\cup\{o\}$, record the tracks occupied by $P$ and $S$ immediately below and immediately above each relevant layer. In an even layer, every horizontal dipath on $A$ follows a consecutive interval of tracks in increasing order. In an odd layer, it follows such an interval in decreasing order.

Suppose that the radial order of $P$ and $S$ reverses for the first time across some layer. If that layer is outward, the path that enters on the inner side must leave beyond the path that enters on the outer side. Their two increasing track intervals then intersect. The same argument applies in an inward layer with the directions reversed. Either case contradicts the disjointness of $P$ and $S$.

A dipath can leave $A$ only through track $0$ in an inward layer. If $P$ is on the outer side of $S$ before it leaves $A$, its horizontal interval to track $0$ would contain the track occupied by $S$ at the beginning of that layer. Thus $P$ meets $S$, a contradiction. Hence a dipath disjoint from $S$ can leave $A$ only from the $o$-side of $S$.

Similarly, a dipath can enter $A$ only through track $0$ in an outward layer. If $P$ enters $A$ and ends that layer on the outer side of $S$, its horizontal interval from track $0$ to its final track contains the track occupied by $S$ at the end of that layer. Thus $P$ meets $S$, a contradiction. Hence a dipath that enters $A$ from $o$ cannot pass to the outer side of $S$.

Therefore a dipath disjoint from $S$ cannot change sides relative to $S$, whether it remains in $A$ or leaves and later re-enters $A$.
\end{proof}

\begin{lemma}\label{lem:Stein-obstruction}
For $a\ge b\ge c\ge1$, the digraph $\widetilde{Y}(a,b,c)$ contains no subdivision of $H_{a+b+2}$.
\end{lemma}

\begin{proof}
Set $s=a+b+2$, and suppose that $\widetilde{Y}(a,b,c)$ contains a subdivision of $H_s$. Let $Q_1,\ldots,Q_s$ be its branch rays, listed in grid order. By Lemma~\ref{lem:no-passing}, we choose a sufficiently large cut $C_t$ so far out that the $Q_i$ use distinct vertical arcs of $C_t$. Every arm contains a track occupied by a branch ray at $C_t$. Indeed, if one arm contained no such track, the centre track and the two remaining arms would contain at most $1+a+b=s-1$ branch rays, a contradiction.

In each arm, choose a branch ray on the outermost occupied track. These three chosen rays are distinct. Since only $Q_1$ and $Q_s$ are end rails, at least one of the three is an internal grid rail. Choose such a ray and call it $R$. Let $A$ be the arm containing $R$. If $A$ contains another branch ray, let $S$ be the branch ray on the second-outermost occupied track of $A$. If $R$ is the only branch ray on $A$, let $S$ be the branch ray on the centre track. Such a centre ray must exist, otherwise $R$ together with the rays on the other two arms would occupy at most $a+b+1=s-1$ tracks. Thus, at the cut $C_t$, the ray $R$ is the only branch ray on the outer side of $S$. The relative positions of R and S are illustrated in Figure~\ref{fig:three-arm}(b).

\emph{Case 1: After $C_t$, the ray $S$ never enters another arm.}
Since $R$ is the unique branch ray on the outer side of $S$ at $C_t$, every branch ray other than $R$ and $S$ lies on the $o$-side of $S$ at $C_t$. By Lemma~\ref{lem:no-passing}, no dipath disjoint from $S$ can cross from the outer side of $S$ to its $o$-side, or conversely. Since $R$ is an internal grid rail, it has two grid neighbours, at most one of them is $S$. Let $Q$ be the other neighbour. Then $Q$ lies on the $o$-side of $S$, so a sufficiently late rung from $R$ to $Q$ would have to cross $S$, a contradiction.

\emph{Case 2: After $C_t$, the ray $S$ first switches from $A$ to another arm $B$.}
Let $h_B$ be the length of $B$. To leave $A$, the ray $S$ reaches $o$ in an inward layer. After possibly waiting on the centre track, it enters $B$ in an outward layer along a segment
\begin{equation}\label{eq:first-switch}
oB_1\cdots B_\rho,
\qquad \rho\ge1.
\end{equation}
Before this switch, $R$ cannot leave $A$ without crossing $S$. If $S$ already lies on the centre track at $C_t$, $R$ is the only branch ray on $A$ by the choice of $S$. Otherwise, consider the inward layer in which $S$ reaches $o$. The horizontal segment of $S$ contains every track between the track occupied by $S$ and $o$. Hence all branch rays on the $o$-side of $S$ must already have left $A$. By Lemma~\ref{lem:no-passing}, no branch ray on the outer side of $S$ can cross $S$. Hence $R$ is the only branch ray left on $A$ when $S$ reaches $o$. While $S$ waits at $o$, no branch ray can enter $A$ through $o$. The same is true in the outward layer in which $S$ follows the segment in ~\eqref{eq:first-switch}, since that segment contains $o$. Thus, $R$ is still the only branch ray on $A$ after this switch.

Consider the cut immediately above the layer containing ~\eqref{eq:first-switch}. The horizontal segment of $S$ uses the lower endpoints of the vertical arcs on
\[
o,B_1,\ldots,B_{\rho-1},
\]
and $S$ itself occupies the track $B_\rho$. None of the tracks $o,B_1,\ldots,B_{\rho-1}$ is available to another branch ray.

Let $C$ be the third arm and let $h_C$ be its length. At this cut, $R$ is the only branch ray on $A$. The ray $S$ uses the vertical arc on the track $B_\rho$. Apart from $S$, at most $h_B-\rho$ branch rays can use the tracks of $B$. At most $h_C$ branch rays can use the tracks of $C$. Therefore the cut contains at most
\begin{equation}\label{eq:track-count}
\underbrace{1}_{R}
+
\underbrace{1}_{S}
+
\underbrace{(h_B-\rho)}_{\text{unused outer part of }B}
+
\underbrace{h_C}_{\text{third arm}}
\le h_B+h_C+1\le a+b+1=s-1
\end{equation}
branch rays. This contradicts the existence of $s$ disjoint branch rays.
\end{proof}

\begin{figure}[H]
\centering
\includegraphics[width=0.92\textwidth]{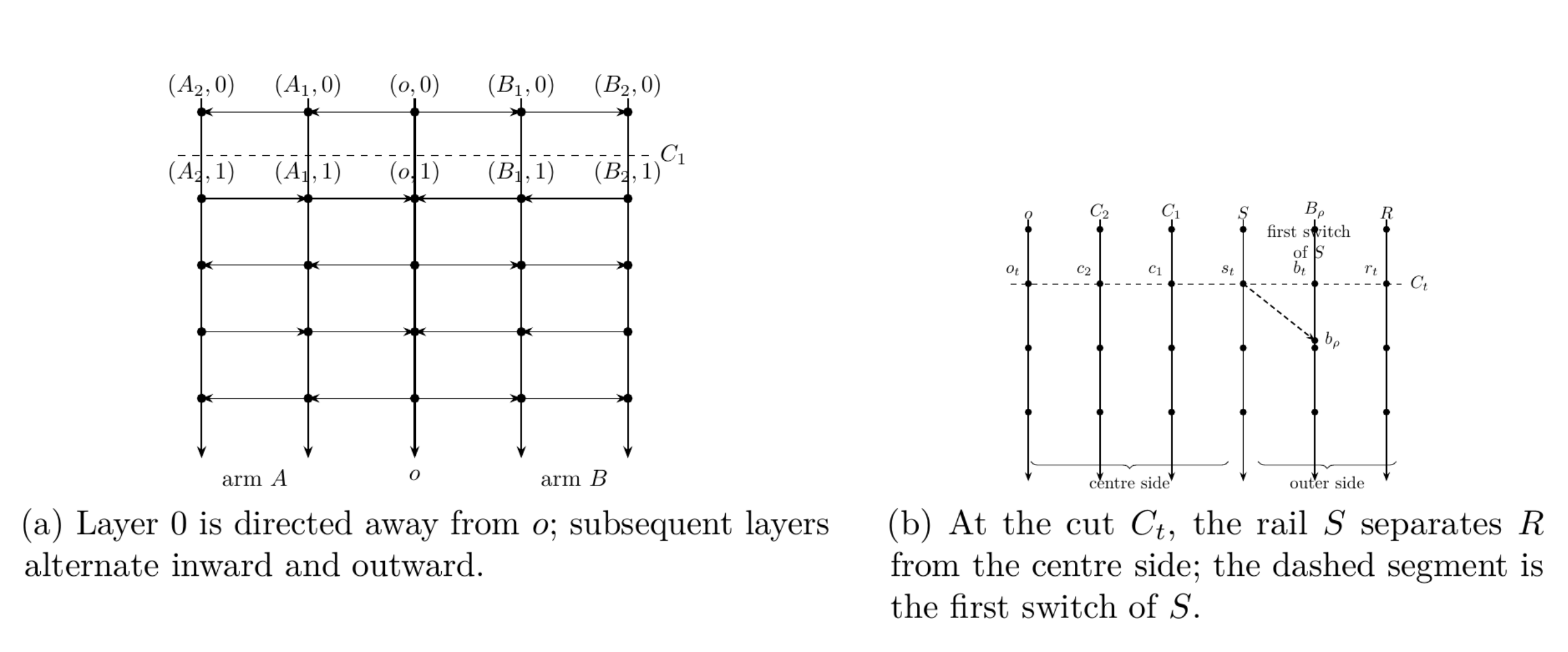}
\caption{The geometry used in Lemmas~\ref{lem:no-passing} and~\ref{lem:Stein-obstruction}. Downward arrows indicate increasing height.}
\label{fig:three-arm}
\end{figure}

\begin{lemma}\label{lem:end-degree-product}
Let $a,b,c$ be three positive integers. $\widetilde{Y}(a,b,c)$ has an end of in-degree $1+a+b+c$.
\end{lemma}

\begin{proof}
The $1+a+b+c$ vertical tracks are pairwise disjoint directed rays. We first show that they lie in one end. Let $R$ and $S$ be two vertical tracks on the same arm. If $R$ is farther from $o$ than $S$, an odd layer contains an $R$--$S$ dipath. If $S$ is farther from $o$, an even layer contains such a dipath. Choosing successively higher layers gives infinitely many pairwise vertex-disjoint dipaths in either direction. The same argument applies when one of the tracks is the centre track. Now suppose that $R$ and $S$ lie on different arms. In an odd layer, move from $R$ to $o$, use the vertical arc at $o$, and in the next even layer move from $o$ to $S$. Repeating the construction in successively higher layers gives infinitely many pairwise vertex-disjoint $R$--$S$ dipaths for either ordered pair of tracks. Reversing the roles of $R$ and $S$ gives the opposite direction. Hence all vertical tracks belong to one end $\omega$.

Suppose that $\omega$ contains $2+a+b+c$ pairwise disjoint directed rays. By Lemma~\ref{lem:no-passing}, every directed ray uses exactly one arc of each sufficiently high cut $C_t$, and pairwise disjoint rays use different cut arcs. Since
\[
|C_t|=1+a+b+c,
\]
no larger family of pairwise disjoint rays exists. The end therefore has in-degree $1+a+b+c$.
\end{proof}

\begin{theorem}\label{thm:lower}
For every $n\ge4$,
\[
k(n)\ge\left\lfloor\frac{3n}{2}\right\rfloor-1.
\]
\end{theorem}

\begin{proof}
For $n=2m$, take $(a,b,c)=(m-1,m-1,m-1)$. By Lemma~\ref{lem:end-degree-product}, the resulting digraph
has an end of in-degree $1+a+b+c=3m-2$. Lemma~\ref{lem:Stein-obstruction} excludes $H_{a+b+2}=H_{2m}$. Thus
$k(2m)\ge3m-1$.

For $n=2m+1$, take $(a,b,c)=(m,m-1,m-1)$. The end now has in-degree $1+a+b+c=3m-1$, while Lemma~\ref{lem:Stein-obstruction} excludes $H_{a+b+2}=H_{2m+1}$. Hence $k(2m+1)\ge3m$.
\end{proof}

\begin{proof}[Proof of Theorem~\ref{thm:main}]
Proposition~\ref{prop:small12} gives the values for widths $1$ and $2$. Equation~\eqref{eq:k3}, obtained from Lemma~\ref{lem:Gamma3} and Theorem~\ref{thm:upper}, gives $k(3)=4$. For every $n\ge4$, Theorems~\ref{thm:upper} and~\ref{thm:lower} give matching upper and lower bounds.
\end{proof}

\begin{remark}
The value $k(3)=4$ is the first place where the vacancy mechanism enters the argument. The recurrence in Theorem~\ref{thm:tree-vacancy} adds one hole and two labels whenever three vertices are added, giving $2m$ labels on $3m-1$ carriers and $2m+1$ labels on $3m$ carriers. Since $h(3)=0$, three carriers provide no vacancy, and the cyclic example $\Gamma_3$ shows that they need not produce the two-way contacts required by $H_3$. For four carriers, $h(4)=1$, so the directed scheduling argument can be applied.
\end{remark}

\section{The forced width as a function of the end-degree}
Theorem~\ref{thm:main} can also be viewed from the opposite direction: for a prescribed end in-degree, one may ask for the largest grid width that is always forced. This leads to the following parameter.

For a positive integer $d$, let $W(d)$ denote the largest integer $n$ such that every digraph $D$ with an end $\omega$ satisfying $d^-(\omega)=d$ contains a subdivision of $H_n$ whose branch rays belong to $\omega$.

\begin{corollary}\label{width correspondence}
For every \(d\ge 4\),
\[
W(d)=\left\lfloor\frac{2d}{3}\right\rfloor+1.
\]
In addition, $W(1)=1, W(2)=2, W(3)=2.$
\end{corollary}

\begin{proof}
An end of in-degree one contains a ray but cannot contain two disjoint branch rays, so $W(1)=1$. If $d=2$, Proposition~\ref{prop:small12} guarantees an $H_2$-subdivision, while an $H_3$-subdivision whose branch rays belong to the same end would contain three pairwise disjoint rays in that end. Thus $W(2)=2$. Every end of in-degree three contains an $H_2$-subdivision, whereas the digraph \(\Gamma_3\) from Lemma~\ref{lem:Gamma3} has an end of in-degree three but contains no $H_3$-subdivision. Hence $W(3)=2$.

Now let \(d\ge 4\), and set
\[
N=\left\lfloor\frac{2d}{3}\right\rfloor+1.
\]

We first prove that \(W(d)\ge N\).
If $d=3m$, then $N=2m+1$, and Theorem~\ref{thm:main} gives $k(N)=k(2m+1)=3m=d.$
If $d=3m+1$, then $N=2m+1$. For $m=1$, we have $N=3$ and $k(N)=k(3)=4=d$; for $m\ge2$, $k(N)=k(2m+1)=3m<d.$
Finally, if $d=3m+2$, then $N=2m+2$, and $k(N)=k(2m+2)=3m+2=d.$ Thus $k(N)\le d$ in every case. Since the given end has in-degree $d\ge k(N)$, Theorem~\ref{thm:upper} yields an $H_N$-subdivision whose branch rays belong to that end. Hence $W(d)\ge N$.

For the reverse inequality, choose
\[
(a,b,c)=
\begin{cases}
(m,m,m-1), & d=3m,\\
(m,m,m),   & d=3m+1,\\
(m+1,m,m), & d=3m+2.
\end{cases}
\]
Since $d\geq 4$, these parameters satisfy
\[
a\ge b\ge c\ge1.
\]
A direct calculation gives
\[
1+a+b+c=d
\qquad\text{and}\qquad
a+b+2=N+1.
\]
By Lemma~\ref{lem:end-degree-product}, the digraph \(\widetilde{Y}(a,b,c)\) has an end of in-degree exactly $d$. By Lemma~\ref{lem:Stein-obstruction}, it contains no subdivision of $H_{a+b+2}=H_{N+1}.$
Hence an end of in-degree $d$ does not necessarily force an $H_{N+1}$-subdivision, and therefore
\[
W(d)\le N.
\]

Combining the two inequalities yields
\[
W(d)=N=\left\lfloor\frac{2d}{3}\right\rfloor+1
\]
for every \(d\ge4\).
\end{proof}

\section*{Declaration on the Use of AI}
AI tools were used at an early stage of this work to assist with checking proof ideas and improving the presentation of the manuscript. All mathematical results and proofs were developed and verified by the authors, who take full responsibility for the content of this work.

\end{document}